\pdfoutput=1
\documentclass{article}
\usepackage{maa-monthly} 
\usepackage{indentfirst}
\usepackage{cite}
\usepackage{amsmath}
\usepackage{amssymb}
\usepackage{dsfont}
\usepackage{graphicx,grffile}
\DeclareGraphicsExtensions{.dot.pdf} 
\graphicspath{{./}{./Figs/}} 
\usepackage{xcolor}
\newtheorem{corollary}{Corollary} 
\newtheorem{theorem}{Theorem} 
\newtheorem{definition}{Definition}  
\newtheorem{lemma}{Lemma}  
\usepackage{algorithm2e} 
\usepackage{multirow}
\usepackage{flushend} 
\usepackage{cuted}
\usepackage{setspace}
\usepackage{subcaption}   
\usepackage{epstopdf}
\usepackage{enumitem} 
\usepackage{listings} 
\lstdefinelanguage{julia}
{
  keywordsprefix=\@, 
  morekeywords={
    exit,whos,edit,load,is,isa,isequal,typeof,tuple,ntuple,uid,hash,finalizer,convert,promote,
    subtype,typemin,typemax,realmin,realmax,sizeof,eps,promote_type,method_exists,applicable,
    invoke,dlopen,dlsym,system,error,throw,assert,new,Inf,Nan,pi,im,begin,while,for,in,return,
    break,continue,macro,quote,let,if,elseif,else,try,catch,end,bitstype,ccall,do,using,module,
    import,export,importall,baremodule,immutable,local,global,const,Bool,Int,Int8,Int16,Int32,
    Int64,Uint,Uint8,Uint16,Uint32,Uint64,Float32,Float64,Complex64,Complex128,Any,Nothing,None,
    function,type,typealias,abstract
  },
  otherkeywords = {julia>},
  morekeywords=[2]{julia>},
  sensitive=true,
  morecomment=[l]{\#},
  morestring=[b]', 
  morestring=[b]" 
}

\usepackage{booktabs} 
\usepackage{dcolumn} 
\newcolumntype{d}[1]{D{.}{.}{#1}}
\usepackage{caption}
\usepackage[draft,colorlinks=true,linkcolor=blue,citecolor=blue,urlcolor=blue,pdfauthor={Matthew Roughan},pdftitle={Rigorous statistical analysis of HTTPS reachability}]{hyperref}

\def\equationautorefname~#1\null{%
  (#1)\null
}

\newlength{\figwidth}
\newcommand{\email}[1]{{\tt #1}}
\newcommand{\eg}{{\em e.g., }} 
\newcommand{\ie}{{\em i.e., }}

\newcommand{\D}{\mathbb D}
\newcommand{\Z}{\mathbb Z}
\newcommand{\Ss}{\mathbb S}

\SetKwInput{KwData}{Input}
\SetKwInput{KwResult}{Output}
\renewcommand{\phi}{\emptyset}

\newcommand{\defequals}{\; \stackrel{\text{def}}{=} \;}

\begin{document} 

\pagestyle{plain}

\title{Surreal Birthdays and Their Arithmetic}

\renewcommand{\thefootnote}{\fnsymbol{footnote}}
\author{
  Matthew Roughan\footnotemark[2]
 \\ \email{\footnotesize matthew.roughan@adelaide.edu.au}
}

\thispagestyle{plain}

\renewenvironment{itemize}{%
  \begin{list}{$\bullet$}{%
    \setlength\labelwidth{1.5em}%
    \setlength{\leftmargin}{1.5em}%
    \setlength{\topsep}{1pt}%
    \setlength{\itemsep}{-0.0mm}%
  }%
  }{\end{list}}

\newenvironment{ssitemize}{%
  \begin{list}{$-$}{%
    \setlength\labelwidth{2em}%
    \setlength{\leftmargin}{2em}%
    \setlength{\topsep}{-1mm}%
    \setlength{\itemsep}{-1mm}%
  }%
  }{\end{list}} 

\newenvironment{sbitemize}{%
  \begin{list}{$\bullet$}{%
    \setlength\labelwidth{2.5em}%
    \setlength{\leftmargin}{1.1em}%
    \setlength{\topsep}{-1mm}%
    \setlength{\itemsep}{-0.5mm}%
  }%
  }{\end{list}}

\renewenvironment{enumerate}{%
   \begin{list}{\arabic{enumi}.}{%
    \setlength\labelwidth{1.5em}%
    \setlength\leftmargin{1.5em}%
    \setlength{\topsep}{4pt plus 2pt minus 2pt}%
    \setlength\itemsep{-0.0mm}%
    \usecounter{enumi}}%
  }{\end{list}}

\maketitle

\renewcommand{\thefootnote}{\fnsymbol{footnote}}
\footnotetext[2]{ARC Centre of Excellence for
          Math. \& Stat. Frontiers (ACEMS), University of Adelaide,
          Adelaide, 5005, SA, Australia}
\renewcommand{\thefootnote}{\arabic{footnote}}
 
\section{Introduction} 

\begin{quote}
  {\em I used to feel guilty in Cambridge that I spent all day playing
  games, while I was supposed to be doing mathematics. Then, when I
  discovered surreal numbers, I realized that playing games IS math.}\\
  \hspace*{7mm} John  Horton Conway
\end{quote}

Surreal numbers, invented by John Horton Conway
\cite{conway2000numbers}, and named by Knuth in ``Surreal Numbers: How
Two Ex-Students Turned on to Pure Mathematics and Found Total
Happiness" \cite{knuth1974surreal}, are appealing because they provide
a single construction for all of the numbers we are familiar with and
many others -- the reals, rationals, hyperreals and ordinals -- with
only the use of set theory. Thus they provide an underpinning idea of
what a number really is. And they have an elegant and complex
structure that is worthy of study in its own right.

Surreal numbers aren't numbers as we are taught in grade school, but
they have many of the same properties. The tricky thing is that they
are defined recursively from the very start.
 
Recursion is like the joke: ``an American, an Englishman and an
Australian walk into a bar, and one of them says 

\hspace{5mm} ``\parbox[t]{0.8\textwidth}{an American, an Englishman and
  an Australian walk into a bar, and one of them says,}

\vspace{1mm} \hspace{14mm} ``\parbox[t]{0.8\textwidth}{an American, an Englishman and
  an Australian walk into a bar, and
  one of them says, ...''} \\[1mm]
%
A recursive definition means that a surreal number is defined in terms
of other surreals and so on. The break in this circularity that makes
it possible to get a foothold is that each turn of the circle
progresses inexorably towards 0, which can be defined {\em ab
  initio}. To continue the Latin the surreal numbers violate the
dictum {\em ex nihlo nihil fit}, or ``from nothing comes nothing.''
From 0 springs forth all other numbers.

The definition that performs this miracle is as follows: a surreal
number $x$ consists of an ordered pair of two sets of surreal numbers
(call them the left and right sets, $X_L$ and $X_R$, respectively)
such that no member of the left set is $\geq$ any of the members of
the right set. We write $x = \{ X_L | X_R\}$ for such a surreal.

This seems a difficult definition. We haven't even defined $\geq$, and
yet are using it in the definition.  Everything resolves because we
can always work with empty sets.  The starting point -- the first
surreal number -- is $\{ \phi | \phi\}$ (where $\phi$ is the empty
set).  A careful reading of the definition says that no elements of
one can be $\geq$ the other, but as there {\em are} no elements, the
comparison is automatically true. We call this number $\overline{0}$.

Then on the ``first day" a new generation of surreals can be created,
building on $\overline{0}$. On the second day we create a second
generation and so on. Each has a meaning corresponding to traditional
numbers in order to have a consistent interpretation with respect to
standard mathematical operators such as addition. The construction is
elegant, and surprisingly general, and leads naturally to the idea of
the ``birthday'' of a surreal numbers being literally the day on
which it is born.
   
A natural question then is, when we perform arithmetic on surreal
numbers, what is the birthday of the result? This paper answers that
question.

\section{The Surreal Numbers and Their Birthdays}
 
We won't describe the surreals in detail here; there are several good
tutorials or books, \eg
\cite{conway2000numbers,knuth1974surreal,tondering13:_surreal_number_introd,grimm12:_introd_surreal_number,14:_proof_conway_simpl_rule_surreal_number,simons17:_meet_the_surreal_number}. In
particular, T{\o}ndering~\cite{tondering13:_surreal_number_introd} and
Simons~\cite{simons17:_meet_the_surreal_number} provide excellent
introductions. We need provide a little background though.  For
instance, notation varies a little: here we denote numbers in lower
case, and sets in upper case, with the convention that $X_L$ and $X_R$
are the left and right sets of $x$, and we write a form as
$x = \{ X_L \mid X_R \}$. I'd like to make a
surrealist/computer-science joke here, namely {\em `{\tt |}' is not a
  pipe}, but a conjunction between Magritte and Unix might be
considered too obscure even for a paper on surreal numbers.  More
prosaically, it is common to omit empty sets, but I prefer writing
$\phi$ explicitly because it is a little clearer when writing
complicated sequences.

In much of the literature the idea of a surreal number is interwoven
with its form.  This is best explained by an analogy to rational
numbers. We can write a rational number in many ways, \eg $1/2 = 2/4$.
That is, we have many \emph{forms} of the same
\emph{number}. Likewise, a surreal number can have many forms. Here,
we work with the form because it is in terms of these that Conway's
surreal arithmetic operators are defined. The distinction requires a
clarification of the notion of equality. Following
Keddie~\cite{keddie94:_ordin_operat_surreal_number} we call two forms
\emph{identical} if they are the same form (\ie have identical
left and right sets), and \emph{equal} if they have the same value
(\ie they denote the same number). We shall distinguish these two
cases by writing equality of value as equivalence, $\equiv$, and
identity by $==$. A single equal sign will be reserved for
conventional numbers.


 



The first surreal number form to be defined is
$\overline{0} \defequals \{ \phi \mid \phi \}$.  We call this number
$\overline{0}$, because it will turn out to be the additive identity
(the 0 of conventional arithmetic).  All other numbers are defined
from this point, following a construction to be laid out below. The
line over the 0 denotes that this is a special, \emph{canonical} form
of zero.

The second two surreal number forms, the numbers we can define on Day
1, immediately after creating $\overline{0}$, are
\[ \overline{1}  \defequals\{\overline{0} \mid \phi\} 
 \;\; \mbox{ and } \;\;
   \overline{- \! 1} \defequals \{\phi \mid \overline{0} \}. 
\]
This notation, however, hides some of the structure of the
surreals. To see them in all their glory we should write
\[ \overline{1}  \defequals \big\{ \{ \phi \mid \phi \} \; \big| \; \phi\big\} 
      \;\;   \mbox{ and }  \;\;
    \overline{- \!1}  \defequals \big\{\phi \; \big| \; \{ \phi \mid \phi \}\big\}, 
\] 
but no doubt you can see that this will quickly result in a very
complicated expressions. We will resolve this by drawing pictures such
as in \autoref{fig:x1}(a). The figure shows each surreal as a node in
a graph. It is (almost) a connected Directed Acyclic Graph (DAG) with
links showing how each surreal is constructed from its
\emph{parents}. But a DAG, by itself, would loose information. The
graph would only specify parents, not left and right parents. So in
displaying the DAG, we show a box for each surreal number, with the
value given in the top section, and the left and right sets shown in
the bottom left and right sections, respectively. From each member of
each set we show a link to its box, and its parents in turn: a red
link indicates a left parent, and blue right. The advantage of the DAG
is that it shows the whole recursive structure of a surreal. 
  
Most aspects of surreals are defined recursively. For instance
$x \geq y$ (which we need even in the definition) means that no
member of $X_L$ is greater than or equal to $y$, and no member of
$Y_R$ is less than or equal to $x$. It might be hard to see how to use
this in the definition when it is also defined in terms of surreals
(which in turn use the definition), but this is the nature of surreal
operations: they are recursive, not just in terms of themselves, but
each definition in turn uses others at lower levels. In any case, it
is now relatively easy to check that $\overline{-1} \leq \overline{0}
\leq \overline{1}$, and we can define further comparisons, for
instance $x \equiv y$ means $x \geq y$ and $y \geq x$. 

Once we have defined $\overline{\pm 1}$, we can proceed to define yet
more surreals. \autoref{fig:x1}(b) shows another form equivalent to
zero, \ie
$\{ \overline{-1} | \overline{1} \} \equiv \{ \phi | \phi \} ==
\overline{0}$.
The graph shows that a value can reappear at multiple places in the
structure of the form: in this case 0 appears both at the top and the
bottom of the DAG.  The information characterising the surreal form is
not its value, or even the values of its subsets, but the structure of
the whole DAG that describes it. So the two `0' nodes in the graph are
different (non-identical) surreal forms that just happen to have the
same value.

\begin{figure}[!t] 
  \centering    
  \begin{tabular}{ccc} 
   \includegraphics[width=0.24\textwidth]{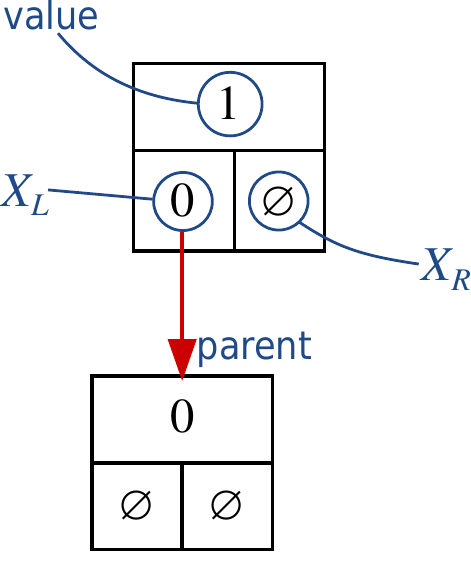}
    \hspace{7mm}
    &  
       \includegraphics[width=0.19\textwidth, trim={12mm 12mm 12mm 12mm}]{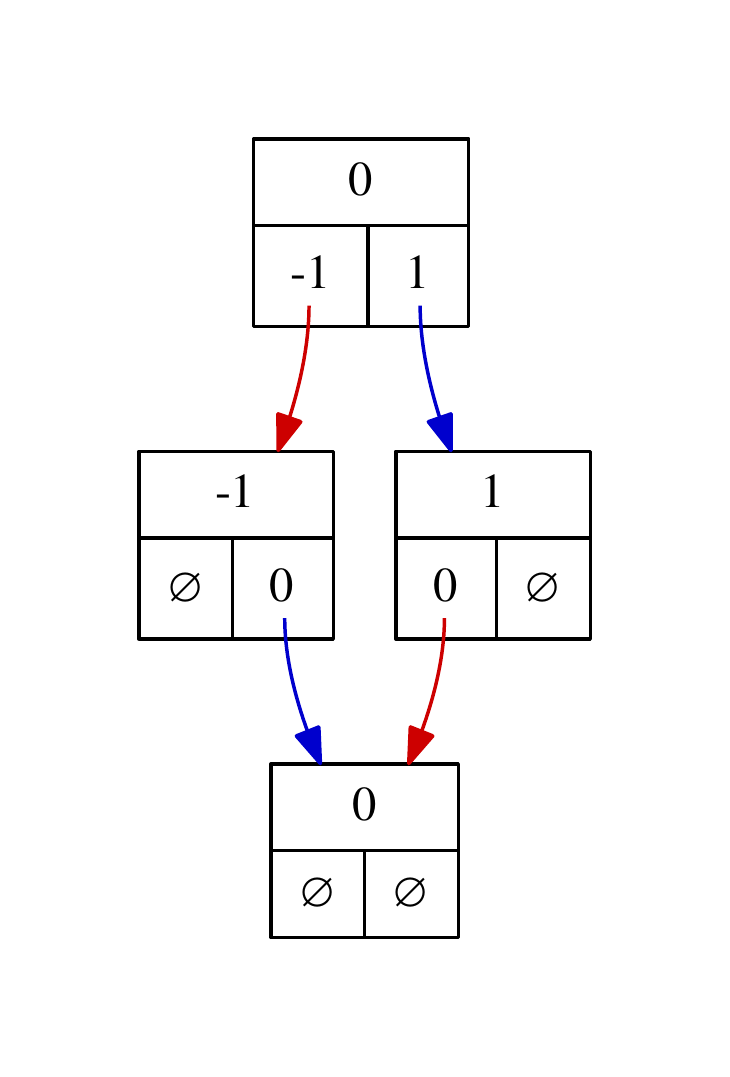}
    \\
    \\ 
   (a) $\overline{1} == \{\overline{0} \mid \phi\}$. &  
   (b) $0 == \{\overline{-1} \mid \overline{1}\}$. \\
  \end{tabular}  
  \caption{DAGs depicting the forms of two surreal numbers. Each box
    represents a surreal number with the value given in the top
    section, and the left and right sets shown in the bottom left and
    right sections, respectively. The arrows show the form of each
    parent surreal, with its own recursive structure. Note that there
    are equivalent forms, \eg
    $\{ \phi \mid \phi \} \equiv \{ \overline{-1} \mid 1 \}$, even within a single
    DAG. }
  \label{fig:x1}
\end{figure}   
 
 
Each number is actually an infinite equivalence class of forms, so we
need to have standard, \emph{canonical} forms, at least to bootstrap
later work.  The standard construction (called the Dali function by
Tondering~\cite{tondering13:_surreal_number_introd}) maps
\emph{dyadic} numbers $\D := \{ n / 2^k \mid n, k \mbox{ integers} \}$ to
(finite) surreals, and is defined recursively by
$d: \D \rightarrow \Ss$ where
\begin{equation}
  \label{eq:dali}
\renewcommand{\arraystretch}{1.2}
 d(x) = \left\{ \begin{array}{ll}
                    \{ \phi \mid \phi \}, & \mbox{if } x = 0, \\
                    \big\{ d(n-1) \mid \phi \big\}, & \mbox{if $x=n$, a positive integer} , \\
                    \big\{ \phi \mid d(n+1) \big  \}, & \mbox{if $x=n$, a negative integer} , \\
                    \left\{ d\left( \frac{n-1}{2^k} \right) \, \right| \left. d\left( \frac{n+1}{2^k} \right) \right\},  \!\!\!\! &
                        \mbox{if } x = n/2^k \mbox { for } k > 0 \mbox{ and } n \mbox{ odd}.
                  \end{array}
                  \right.  
\end{equation}
For convenience, we denote canonical forms through the shorthand of
placing a line above the number. All other forms are defined by their
DAG, or in terms of canonicals.

This recursive construction is often illustrated as a tree
\cite{2016arXiv160803413M,grimm12:_introd_surreal_number,preston14:_surreal,ehrlich11:_conway}
showing the numbers that are created in each generation and their
position on the real number line. However, that is misleading. The
(non-integer) Dali surreals have two parents, and the resultant
structure of dependency in the recursion is the DAG shown in
\autoref{fig:dyadic_tree}. We can see from this, for instance, that
each canonical form has either one or two parents. 


The construction of surreals leads to the notion of \emph{birthdays}:
take $\overline{0}$ to be born on Day 0, and $\overline{\pm 1}$ to be
born on Day 1, and so on, then we can assign a birthday to all
surreals.  I prefer the term \emph{generation} over birthday if only
because it links up to the notion of parents and children more
cleanly.  The birthday can also be seen as how deeply you must recurse
through the DAG structure to get to $\overline{0}$.  We say a surreal
is older if it comes from an earlier generation, \ie it has a smaller
(earlier) birthday. We formalise notions these below.
 
\begin{figure*}[!t] 
  \centering
  \includegraphics[width=1.1\textwidth]{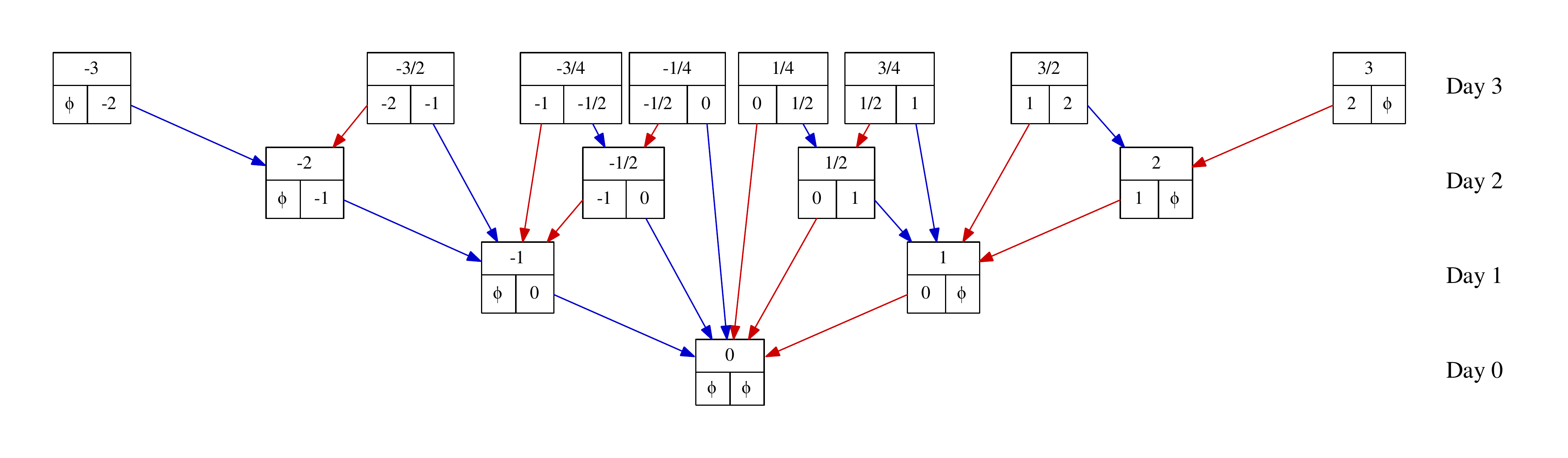}
  \caption{The dyadic DAG, \ie the recursive structure of the dyadic
    surreal numbers, up to birthday/generation 3.}
  \label{fig:dyadic_tree}
\end{figure*}
 
\begin{definition}
  We refer to the elements of the left and right sets of a surreal
  form $x$ as its \emph{parents} (note there may be other than two
  parents), and $x$ as their \emph{child}.  The \emph{birthday} or
  \emph{generation} of $\{\phi \mid \phi \}$ is 0, and the generation
  of all other finite surreal forms is 1 greater than that of their
  youngest parent.
\end{definition}


\noindent If we denote the parents of $x$ by $X_P = X_L \cup X_R$,
then the generation/birthday function of a surreal number
form $x$ is given by
\begin{equation}   
  \label{eq:rho_def}
  g(x) = \sup_{x_p \in X_P} g(x_p) + 1,
\end{equation}  
where $g( \{ \phi \mid \phi \} ) = 0$.  
For example, in \autoref{fig:x1}(b)
$g\big(\{\overline{-1} \mid \overline{1} \}\big) = g(\overline{1}) + 1 = 2$.

Surreal forms with equivalent values can come from different
generations, so knowing the value of a surreal tells us only a lower
bound on its generation (namely the generation of the canonical form
of that surreal). Thus, some questions arise in regard to birthdays:
\eg can we derive birthdays for standard surreal constructs?

We call these problems in \emph{birthday arithmetic}. 

We will start by deriving the birthday of the canonical forms as it
provides an example of the standard proof structure for many surreal
arguments. Simons  states the result
\cite[p.27]{simons17:_meet_the_surreal_number} but only as a minor
note within a larger result.  As in many proofs in this domain, it is
inductive. Throughout this we use $g(x)$ as shorthand for $g(d(x))$.

\begin{lemma}
  \label{lem:canonical}
  The generation/birthday of the canonical form of dyadic $x = n/2^k$,
  which is in irreducible form (or lowest terms) is
  \[ g(x) = \big\lceil \, |x| \, \big\rceil + k, \]
  where $\lceil x \rceil$ denotes the ceiling function of $x$ (the
  smallest integer larger than $x$).
\end{lemma}

\begin{proof}
  The statement is true for $x = \overline{0}$ because
  $g(\overline{0}) = 0$ by definition. The negative case can be
  treated by considering $g(-x)$ (see \autoref{sec:negation}), so we
  only consider $x>0$ here. The integer case is trivial (see
  \autoref{fig:dyadic_tree}) so we focus on the case $n$ odd and
  $k>0$.  

  Assume for the purpose of induction that the lemma is true for all
  parents of $x$.

  From Definition~\autoref{eq:dali} such a dyadic has exactly two
  parents, and hence \autoref{eq:rho_def} reduces to
  \[ 
  g\left(  \frac{n}{2^{k}} \right) 
    = \max\left\{ g\left( \frac{n-1}{2^{k}}\right), g\left(  \frac{n+1}{2^{k}} \right) \right\}  + 1.
  \] 
  For $n$ odd and $k>0$ both $x=n/2^{k}$ and $(n+1)/2^{k}$ have the
  same ceiling (call it $m$), and $\lceil (n-1)/2^{k} \rceil \leq m$,
  so by the inductive hypothesis
  \[
  g\left(  \frac{n}{2^{k}} \right) 
  =  g\left(  \frac{n+1}{2^{k}} \right)  + 1 .
  \]
  Now $n+1$ is even so we can simplify $(n+1)/2^k = \ell/2^{k-1}$, and
  by the inductive hypothesis the theorem is true for the parents of
  $x$, so we must have
  \[  g\left(  \frac{n+1}{2^{k}} \right) = m + k - 1. \]
  Hence $ g\left(n/2^{k} \right) = m + k$. 
\end{proof}
 
Intrinsic to this proof (and others) is the fact that $\overline{0}$
is the starting point for the construction of the surreals, and
therefore is the ultimate ancestor of all surreals.

\section{Addition and Subtraction}
 
In order for the surreals to fulfil their role as ``numbers'' they
must be able to play all the tricks of numbers, for instance, we must
be able to do arithmetic. Conway defined addition and subtraction for
surreal numbers, and showed these satisfy the conditions required, but
they are actually operations on the forms. Let us examine them in
detail below.
 
\subsection{Addition} 
 
The standard definition of addition on surreal forms
\cite{conway2000numbers,tondering13:_surreal_number_introd} is  
\begin{eqnarray*}
  x + y      &  \defequals  & \{ X_L + y \cup x + Y_L \mid X_R + y \cup x + Y_R \}.
\end{eqnarray*}
Notation is often abbreviated, and so you will sometimes set
operations simplified, \eg $\{ x, y \} \cup A$ is written
$\{x, y, A\}$. Also, in this and other definitions we implicitly
extend the operators to sets, or combinations of sets and set with
surreals, \eg $x+y$ is comprised of terms like $X_L + y$.  Operations
on sets are applied to each member
\cite{tondering13:_surreal_number_introd}:
\begin{eqnarray*}
  \{ x_1, x_2, \ldots, x_n \} + y & \defequals  & \{ x_1+y, x_2+y, \ldots, x_n+y \},
\end{eqnarray*}
and operations on empty sets result in empty sets, \ie $x + \phi = \phi$.

Many of the texts on surreals provide proofs that addition satisfies
all of the usual requirements, \eg associativity, commutativity and so
on \eg
\cite{conway2000numbers,tondering13:_surreal_number_introd}. For
instance, 0 is the additive identity
\cite{conway2000numbers,tondering13:_surreal_number_introd}
\[
  x + 0  \equiv 0 + x \equiv x, \\
\]
but this is a statement about values not forms: $x+0$ is not (in
general) identical to x, except for the canonical zero, \ie
$x + \overline{0} == x$.  For instance, consider the following
addition of $\overline{1/2} + 0$ noting carefully the bars indicating
which terms are canonical.
\begin{eqnarray*}
 \{ \overline{0} | \overline{1} \} + \{ \overline{-1} \mid \overline{1} \} 
   & == & \{ \overline{0} + 0, \;  \overline{1/2} +  \overline{(-1)}
          \mid  \overline{1} + \overline{0}, \; \overline{1/2} +
          \overline{1} \} \\ 
   & == & \{ 0, -1/2 \mid \overline{1}, 3/2 \},
\end{eqnarray*} 
which is not identical to $\{ \overline{0} | \overline{1} \}$. 






Another instructive example is $\overline{2} + \overline{2}$. In this
case $X_R = Y_R = \phi$ and so the right-set of
$\overline{2} + \overline{2}$ will also be $\phi$, thus
\begin{eqnarray*}
  \overline{2} + \overline{2}
        & == & \{ \overline{1} + \overline{2}, \overline{2} + \overline{1} \mid \phi \} \\
        & == & \{ \overline{3} \mid \phi \}  \\
        & == & \overline{4}.
\end{eqnarray*}
The result is the canonical form of 4. Naively, we might expect that
addition of canonical forms would always lead to the same. However
this is not true. We can play with such hypotheses using the {\tt
  SurrealNumbers} package\footnote{The SurrealNumbers toolkit (v0.1.1)
  is released under the MIT license, and available at
  \url{https://github.com/mroughan/SurrealNumbers.jl}} in the
programming language Julia. For instance, the above calculation can be
performed using the commands
\begin{lstlisting}[escapeinside={(*}{*)}]
  julia> using SurrealNumbers 
  julia> x = dali(2) (* \em // set x to the canonical form of 2 *)
  julia> y = x + x     (* \em \;\; // calculate $\overline{2} + \overline{2}$ *)
\end{lstlisting}  
A more complicated example with a non-canonical result is
\[ \overline{1} + \overline{1/2} == \{   \{   \{   \phi   \mid   \phi   \}   \mid   \{   \{   \phi   \mid   \phi   \}   \mid   \phi   \}   \} , \{   \{   \phi   \mid   \phi   \}   \mid   \phi   \}   \mid   \{   \{   \{   \phi   \mid   \phi   \}   \mid   \phi   \}   \mid   \phi   \}   \}.\]
\autoref{fig:addition} shows this, and two other forms with value
3/2. The figure makes it easy to see that the latter two forms are
non-canonical, and we see in (c) a different generation as well.
 
\begin{figure}[!t] 
  \centering    
  \vspace{-5mm}
  \begin{subfigure}[b]{0.29\columnwidth} 
    \centering    
    \includegraphics[height=1.5\textwidth, trim={10mm 10mm 10mm 10mm}]{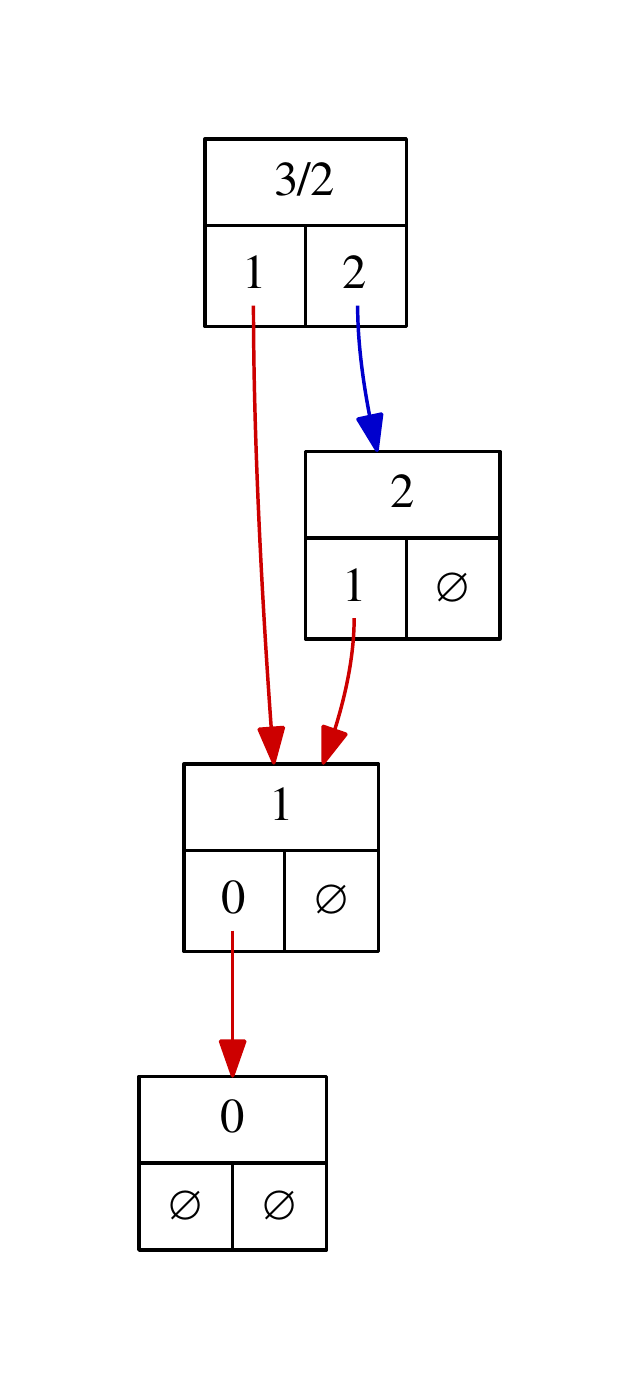} 
    \caption{Canonical form $\overline{3/2}$.}
  \end{subfigure}
  \hspace{0.01\columnwidth} 
  \begin{subfigure}[b]{0.29\columnwidth}
    \centering    
    \includegraphics[height=1.5\textwidth, trim={10mm 10mm 10mm 10mm}]{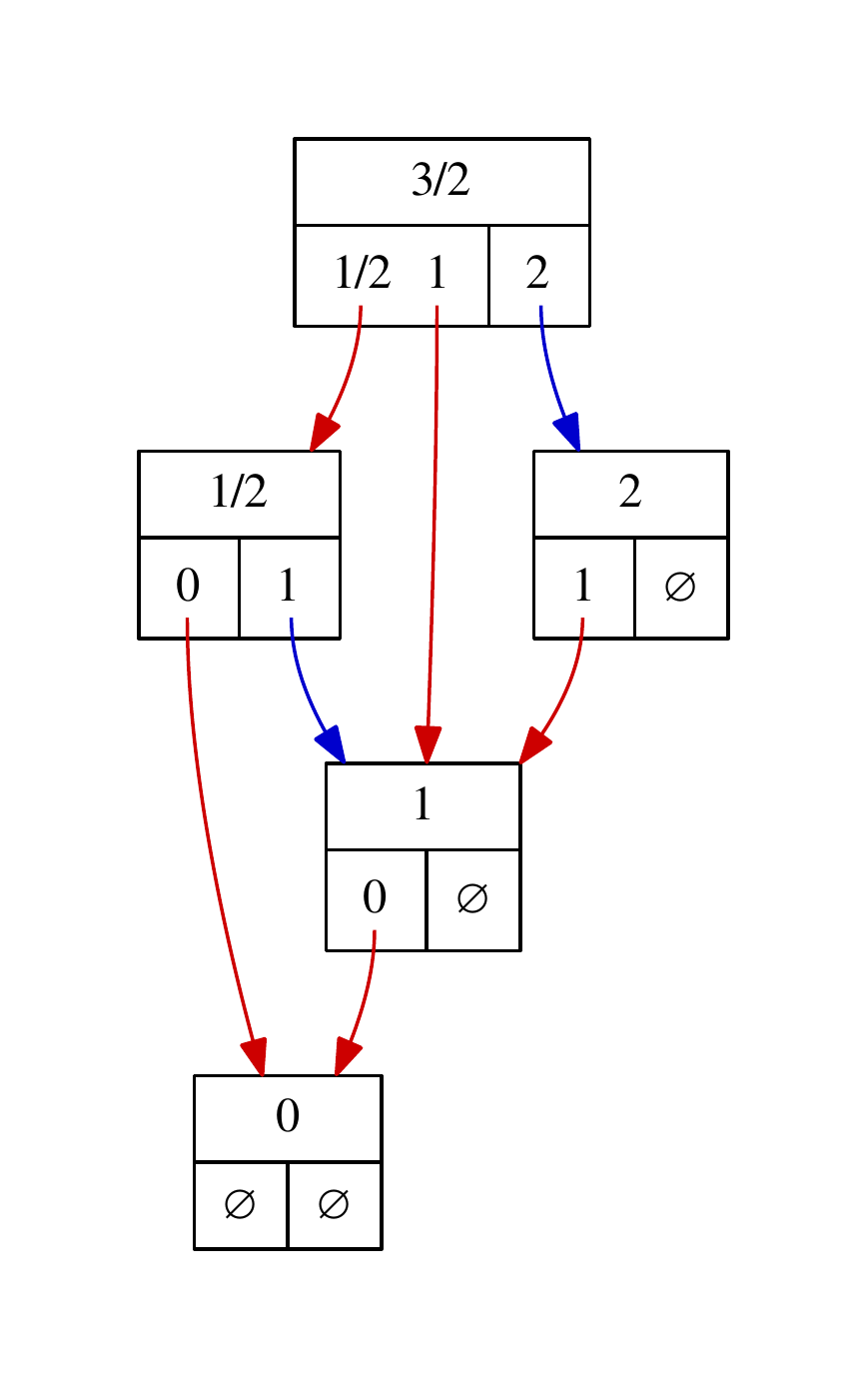} 
    \caption{The form of $\,\overline{1}$ + $\overline{1/2}$.}
  \end{subfigure}
  \hspace{0.01\columnwidth}  
  \begin{subfigure}[b]{0.35\columnwidth}
    \centering    
    \includegraphics[height=1.9\textwidth, trim={10mm 10mm 10mm 10mm}]{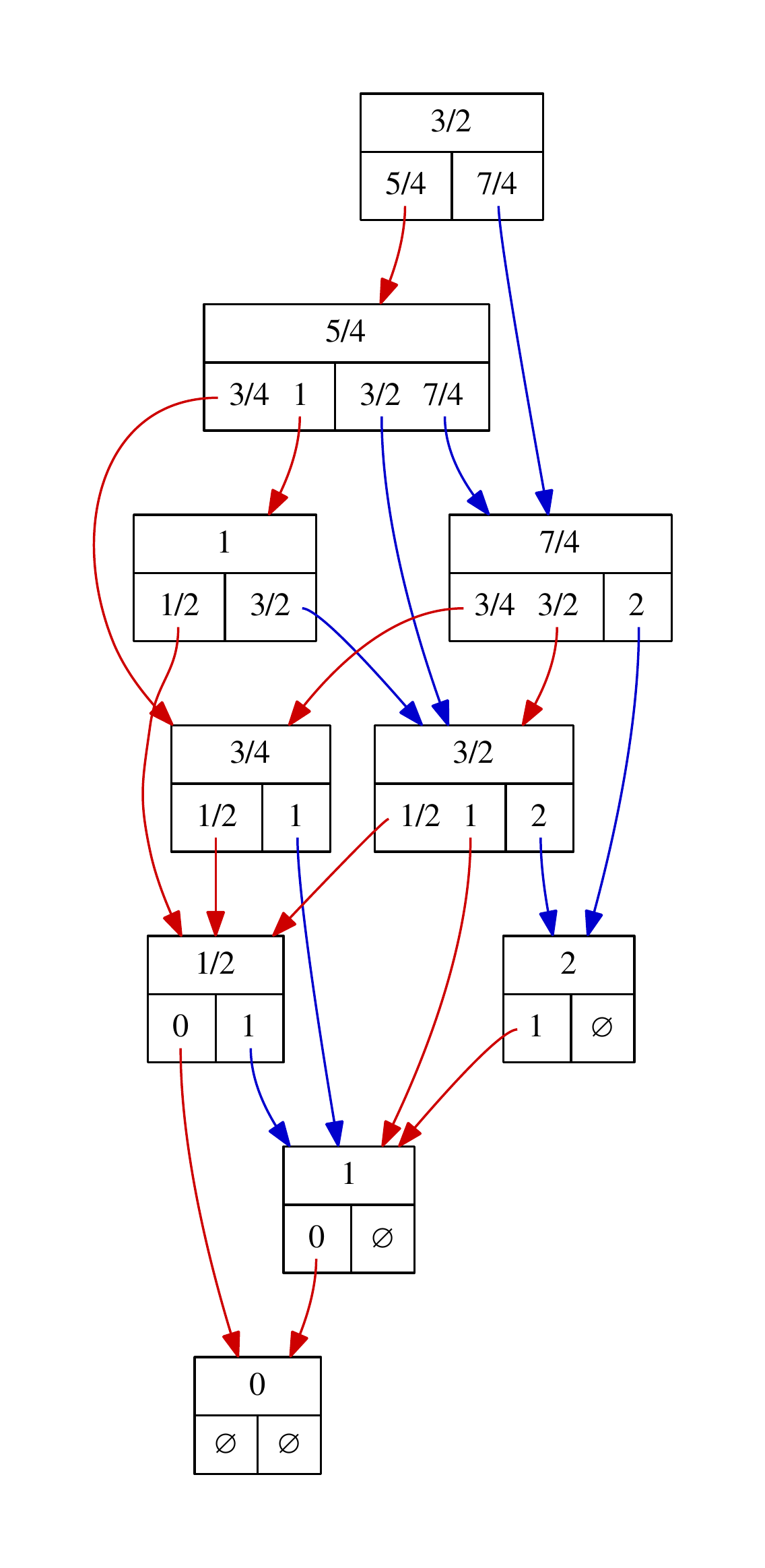} 
    \caption{The form of $\,\,\overline{3/4}$ + $\overline{3/4}$.}
  \end{subfigure}
  \caption{Graphs depicting three forms for the surreal number 3/2. It
    is interesting that the form in (b) appears as a subgraph of the
    form in (c). An open question concerns whether this is true in
    general: is the canonical form always a subgraph of any
    alternative form of the same number?}
  \label{fig:addition}
\end{figure}

That brings us to the nub of the problem -- can we calculate the
generation/birthday of the sum purely from the generation of the
inputs? Simons
proves~\cite[p.25]{simons17:_meet_the_surreal_number}
that $g(x+y) \leq g(x) + g(y)$, but we present a small set of examples
in \autoref{tab:addition_birthday}, and in all of these (and every
other case tested) we find $g(x+y) = g(x) + g(y)$. The question then
is, can we prove that equality is always the case? The answer
follows. 
 
\begin{table}[th]
   \centering
   \renewcommand{\arraystretch}{1.2}
   \begin{tabular}{cc|rrrrr|l}
     \hline
       \abrule
       &    & $\overline{0}$ & $\overline{1/2}$ & $\overline{3/4}$ & $\overline{1}$ & $\overline{2}$ & $x$     \\
    $y$  & $g(y)$   & 0 & 2 & 3 & 1 & 2 & $g(x)$  \\
       \hline
       \abrule
      $\overline{0}$   & 0 &   0  & 2  & 3 & 1  & 2 \\
      $\overline{1/2}$ & 2 &   2  & 4  & 5 & 3  & 4 \\
      $\overline{3/4}$ & 3 &   3  & 5  & 6 & 4 & 5  \\
      $\overline{1}$   & 1 &   1  & 3  & 4 & 2  & 3 \\
      $\overline{2}$   & 2 &   2  & 4  & 5 & 3  & 4 \\
    \hline
       \brule
   \end{tabular}
   \caption{The birthday table $g(x+y)$, \ie the birthdays of the
     sums of the $x$ and $y$ values in the columns and rows. Note that
     they are additive, \ie $g(x + y) = g(x) + g(y)$. }
   \label{tab:addition_birthday}
\end{table} 

\begin{theorem}[Birthday addition theorem]
  \label{thm:birth_add}
  For two surreal numbers $x$ and $y$
  \[ g(x + y) = g(x) + g(y), \]
  where $g(\cdot)$ is the birthday/generation function. 
\end{theorem}
\begin{proof}
  Birthday addition is trivially true if $x$ or $y = \overline{0}$.
  Presume that the theorem is true for all combinations of the parents
  of summands $x$ and $y$ and the summands themselves (excepting $x+y$
  itself). Apply Definition \autoref{eq:rho_def} to addition and we
  get
  \[
    g(x+y) = \sup_{x_p, y_p} \Big\{ g(x + y_p) + 1,  g(y + x_p) + 1 \Big\}. 
  \]
  Now, by the inductive hypothesis, birthday addition is true for all
  parents and their combinations so the above reduces to 
  \begin{eqnarray*}
    g(x+y) 
     & = & \sup_{x_p, y_p} \Big\{ g(x) + g(y_p) + 1, g(y) + g(x_p) + 1 \Big\} \\
     & = & \max\Big\{ g(x) + \sup_{y_p} \{ g(y_p)\} + 1, g(y) + \sup_{x_p} \{ g(x_p)\}  + 1 \Big\}  \\
     & = & \max\Big\{ g(x) + g(y), g(x) + g(y) \Big\} \\
     & = & g(x) + g(y).
  \end{eqnarray*}
  Thus by induction the result. 
\end{proof}

\subsection{Negation and subtraction} 

Negation and subtraction are defined together as 
\[ -x \defequals  \{ -X_R \mid - X_L \}
    \;\; \mbox{and} \;\;
   x - y \defequals x + (-y).
\]
Subtraction looks as simple as addition, but can complicate matters
more than one might think with respect to surreal forms because the
number of empty sets in the results change, and we end up with more
complicated expressions. For instance $\overline{1} - \overline{1}$ is
equivalent to 0 (this is the form that is illustrated in
\autoref{fig:x1}(b)), but is not identical to $\overline{0}$. Thus
$-x$ is the additive inverse of $x$ in terms of equivalence, but not
identity. 

However, here we are interested in the birthday/generation of the
output. We can construct a very simple inductive proof that
$g(-x)=g(x)$, using the definition above. Incidentally, this concludes
the proof of \autoref{lem:canonical}. 
\label{sec:negation}

From this, the definition of subtraction and \autoref{thm:birth_add}
we immediately get the following corollary.

\begin{corollary}[Birthday subtraction corollary]
  \label{thm:birth_sub}
  For two surreal numbers $x$ and $y$
  \[ g(x - y) = g(x) + g(y), \]
  where $g(\cdot)$ is the birthday/generation function. 
\end{corollary}

\section{Multiplication}
 \label{sec:mult}

If you thought addition and subtraction were complicated then fasten
your seat belts. Multiplication is defined by
\begin{multline}
  \label{eqn:multiplication}
  x y   \defequals \Big\{ 
                               \{ X_L \, y + x Y_L - X_L Y_L\} \cup
                               \{X_R \, y + x Y_R - X_R Y_R\}
                                \; \Big| \;\\
                                \{X _L \, y + x Y_R - X_L Y_R\} \cup
                               \{ X_R \, y + x Y_L - X_R Y_L\}                                
                            \Big\}.  
\end{multline}
We need to be clear about exactly what each term in this definition
means because it \emph{does not} follow the typical convention for
expressions of this form. Consider the term
$\{ X_L \, y + x Y_L - X_L Y_L\}$; this means create a set by
taking all pairs $x_i \in X_L$ and $y_j \in Y_L$ and combining as
follows:  
\[ \{ X_L \, y + x Y_L - X_L Y_L\} \defequals
   \{ x_i y + x y_j - x_i y_j \mid
  \forall x_i \in X_L \mbox{ and } y_j \in Y_L \},
\]
where each of the products in the set above is another surreal
multiplication and the additions are surreal additions. 
 



Another quick example is informative: let's work quickly through
$\overline{2} \times \overline{2}$. Once again $X_R = Y_R = \phi$, and so only one term of
the four in the multiplication is non-empty:  
\begin{eqnarray*}
  \overline{2} \times \overline{2}
      & == & \{ \overline{1} \times \overline{2} + \overline{2} \times \overline{1} - \overline{1} \times \overline{1} \mid \phi \} \\ 
      & == & \{ \overline{2} + \overline{2} - \overline{1} \mid \phi \} \\
      & == & \{ \overline{4} - \overline{1} \mid \phi \} \\
      & \equiv & \{ 3 \mid \phi \},
\end{eqnarray*}
where we exploit the fact that $\overline{1}$ is the multiplicative
identity for surreal forms, and $\overline{0}$ is the multiplicative
annihilator, \ie $\overline{0} x = \overline{0}$ for all $x$, a
fact seen by considering that
\[
  \overline{0} x
     == \{ \phi \, y + \overline{0} Y_L - \phi Y_L, 
           \phi \, y + \overline{0} Y_R - \phi Y_R \mid 
           \phi \, y + \overline{0} Y_L - \phi Y_L, 
           \phi \, y + \overline{0} Y_R - \phi Y_R \},
\]
and noting that operations with $\phi$ result in $\phi$. The proof of
the identity of $\overline{1}$ is very similar though it requires an
inductive step. 

We also use a minor result we have not bothered to prove that addition
of canonical forms of non-negative integers results in canonical
forms, however, subtraction does not. The result of
$\overline{2} \times \overline{2}$ looks like the canonical form for
4, but we are taking liberties by reducing `3' to its short hand. This
isn't the canonical 3, and so the shorthand is here misleading. In
fact if we write this out in full we get (as in
\cite{markcc06:_arith_surreal_number})
\begin{multline*}
  \overline{2} \times \overline{2}
  == \{   \{   \{   \{   \{   \{   \phi   \mid   \{ 
        \phi   \mid   \phi   \}   \}   \mid   \{   \{
        \phi   \mid   \phi   \}   \mid   \phi   \}   \}
        \mid   \{   \{   \{   \phi   \mid   \phi   \}
        \mid   \phi   \}   \mid   \phi   \}   \}   \mid  \\ 
    \{   \{   \{   \{   \phi   \mid   \phi   \}   \mid
       \phi   \}   \mid   \phi   \}   \mid   \phi   \}   \}
       \mid   \{   \{   \{   \{   \{   \phi   \mid   \phi
       \}   \mid   \phi   \}   \mid   \phi   \}   \mid
       \phi   \}   \mid   \phi   \}   \}   \mid   \phi   \}.
\end{multline*}

Gonshor proves a Weak Birthday Multiplication Theorem
\cite[Theorem~6.2]{gonshor_1986}, namely that
$g(x y ) \leq 3^{g(x) + g(y)}$.
Simons
\cite{simons17:_meet_the_surreal_number} conjectures a much tighter
bound that $g(x y) \leq g(x) g(y)$, and states that there are no
obvious counter-examples. It was this hypothesis that largely
motivated this investigation.
 
The problem with Simons' statement is that there are very few examples
at all. Multiplication is very complex to do in practice. As far as I
am aware only a few multiplications are explicitly tabulated!  Several
places show that simple multiplicative identities hold, but few go any
further. With the help of the {\tt SurrealNumbers} package we can
calculate other products (they quickly become too complicated to do
with pen and paper). The result of $\overline{2} \times \overline{3}$
is shown in DAG form in \autoref{fig:mult}. From this we can calculate
$g(\overline{2} \times \overline{3}) = 12$. Unfortunately this breaks
Simons' conjecture, which suggests the bound should be 6. It also
suggests that Gonshor's bound is very weak.

\autoref{tab:gen_mult} shows a table for the birthday/generation of
products. The black numbers are those that we have calculated (in
multiple ways) using the Julia package. Note that the only surreal
forms with birthdays 0 and 1 are $\overline{0}$ and $\overline{\pm 1}$
so the first two rows and columns of this table are trivially true.

Patterns appear in the results, \eg
$g(\overline{2} \times \overline{m}) = m (m+1)$, but the general
pattern is not so trivial.  For instance, the product
$g(\overline{3} \times \overline{3}) = 31$ is prime, and so the
resulting birthdays are not even the product of a function the
underlying birthdays!
   
The results below explain the pattern observed in the black values of
the table, and can be extrapolated to provide the gray values. 

\begin{figure}[!t]  
  \centering    
  \includegraphics[width=1.0\textwidth]{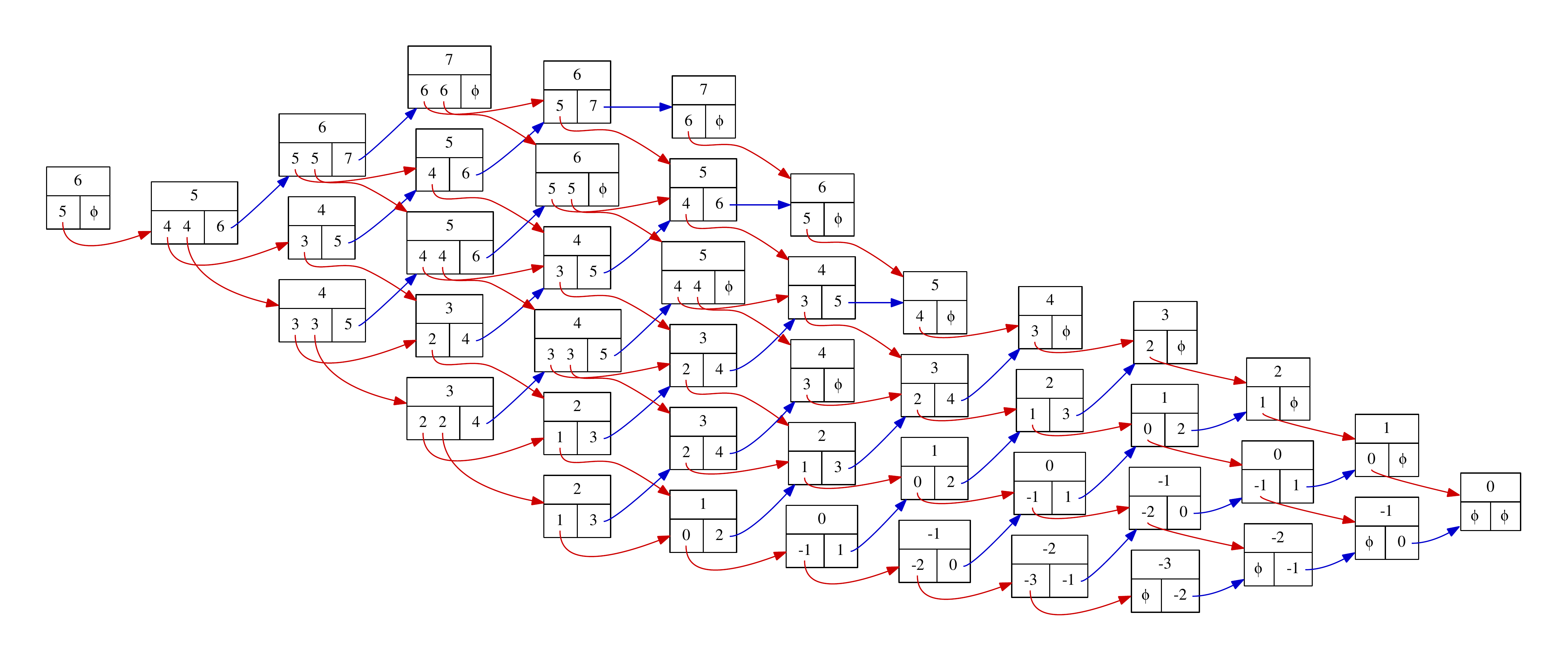} 
  \caption{Graph depicting the form of
    $\overline{2} \times \overline{3}$. Inside this we frequently see
    subforms that might superficially appear identical, \eg
    $4 = \{ 3 \mid 5\}$ appears three times, but these are not 
    identical forms, which is implicit in their dependent
    DAG. Interestingly, the DAG also includes even some negative
    values, which might be unexpected in a product of positive
    integers.}
  \label{fig:mult}
\end{figure}

\begin{lemma}
  If the generation of $x y$ is given by a function $f$
  of the generations of $x$ and $y$, \ie
  $g(x y) = f\big( g(x), g(y) \big)$, then the function will be
  symmetric, \ie $f(n,m) = f(m,n)$, and strictly increasing. 
\end{lemma}

\begin{proof}
  The function must be symmetric by the commutativity of
  multiplication. Now find the ``maximal'' parent of $x$, \ie that is
  the parent $x_p^{(max)}$ that has the maximum generation so that
  $g(x) = g(x^{(max)}_p) + 1$. If the generation of $x y$ is given by
  a function $ f\big( g(x), g(y) \big)$ then 
  \[ g(x y) = f\big(g(x^{(max)}_p) + 1, g(y)\big).
  \]
  Also $x y$ contains parents containing products of all pairs of
  parents of $x$ and $y$, and so $ g(x y) > g(x_p y)$ for all
  $x_p \in X_P$. Choosing the maximal parent we get
  \[  f\big(g(x^{(max)}_p) + 1, g(y)\big) >  f\big(g(x^{(max)}_p), g(y)\big),
  \]
  and as we can find examples $x$ for any $n=g(x)$, this must be true
  for all values of $n$, and likewise $m$, and hence the function is
  increasing. 
\end{proof}

\begin{table}[t]
  \centering
  \footnotesize
  \setlength{\tabcolsep}{2.5mm}
  \renewcommand{\arraystretch}{1.2}
  \begin{tabular}{rr|rrrrrrr|l}
    \hline
           & & \multicolumn{7}{|c|}{$m = g(y)$} & \\
           & & 0 & 1 & 2 & 3 & 4 & 5 & 6 &  \\
    \hline
     \multirow{7}{*}{$n = g(x)$}
    &  0  & 0 & 0 & 0  &   0 &  0 &   0 &  0 \\  
    &  1  & 0 & 1 & 2  &   3 &  4 &   5 &  6 & $ = m$ \\
    &  2  & 0 & 2 & 6  &  12 & 20 &  30 & 42 & $ = m(m+1)$ \\
    &  3  & 0 & 3 & 12 &  31 & 64 & 115 & {\color{gray} \it 188} & \\ 
    &  4  & 0 & 4 & 20 &  64 & {\color{gray} \it 160}  & {\color{gray} \it 340} & {\color{gray} \it 644} & \\ 
    &  5  & 0 & 5 & 30 & 115 & {\color{gray} \it 340} & {\color{gray} \it 841} & {\color{gray} \it 1826} &     \\
    &  6  & 0 & 6 & 42 & {\color{gray} \it 188} & {\color{gray} \it 644} & {\color{gray} \it 1826} & {\color{gray} \it 4494} &   $f(n,m)$    \\
    \hline
       \abrule
  \end{tabular}
  \caption{The values of $g(x y)$ in relation to $n=g(x)$ and
    $m=g(y)$. Roman values have been verified through multiple
    instances of the products of (not just canonical) forms from the same
    generation. 
    Italic are extrapolated using \autoref{thm:birthday_mult}. The
    right-hand column expresses the pattern, where known. }
  \label{tab:gen_mult} 
\end{table}

\begin{theorem}[Birthday multiplication theorem]
  \label{thm:birthday_mult}
  The generation of $x y$ is given by a function
  $f(\cdot,\cdot)$ of the generations of $x$ and $y$, \ie
  $g(x y) = f\big( g(x), g(y) \big)$, where the function
  $f(\cdot,\cdot)$ satisfies the recurrence relation
  \[
     f(n,m) = \left\{ \begin{array}{ll}
                        0, & \mbox{ if } m=0, \\
                        0, & \mbox{ if } n=0, \\
                        f(n,m-1) + f(n-1,m) + f(n-1,m-1) +1, & \mbox{ otherwise}, \\
                      \end{array}
               \right.      
  \]                   
  for $n, m \in \Z^+$.
\end{theorem}

\begin{proof}
  The only number $x$ with $g(x) = 0$ is $\overline{0}$ which is the
  multiplicative annihilator, \ie
  $\overline{0} \times x = \overline{0}$ for all $x$, and hence the
  bounding case $n=0$ and by symmetry $m=0$.

  For $n,m > 0$ we start from the definition of multiplication
  \autoref{eqn:multiplication}, which contains terms
  $x_p y + x y_p - x_p y_p$ for all pairs of parents $(x_p, y_p)$,
  and hence
  \begin{eqnarray*}
    g(x y) 
    & = & \sup_{(x_p, y_p)}   g(x_p y + x y_p - x_p y_p) + 1 \\
    & = & \sup_{(x_p, y_p)}\left[  g(x_p y) + g(x y_p) + g(x_p y_p) \right] + 1,
  \end{eqnarray*}
  by the Birthday Addition Theorem, and its Subtraction corollary. 

  Assume (for inductive purposes) that
  $g(x_p y) = f\big(g(x_p), g(y)\big)$, and similarly for the other
  such terms, and so
  \begin{eqnarray*}
    f\big( g(x), g(y) \big)
    & = &  \sup_{(x_p, y_p)}\left[  f\big(g(x_p), g(y)\big) + f\big(g(x), g(y_p)\big) + f\big(g(x_p), g(y_p)\big) \right] + 1,
  \end{eqnarray*}
  We can choose the respective parents $x_p$ and $y_p$
  independently. By the previous theorem the function $f(n,m)$ must be
  increasing. Hence the above sum will be maximised when we choose
  $(x_p, y_p)$ such that $g(x_p)$ and $g(y_p)$ are both individually
  maximised.  In this case note the definition of $g(x)$ in
  \autoref{eq:rho_def},
  and hence 
  \begin{eqnarray*}
    \lefteqn{f\big( g(x), g(y) \big)} \\
    & = &  \sup_{x_p}  f\big(g(x_p), g(y)\big) +  
          \sup_{y_p} f\big(g(x),g(y_p)\big) +     \sup_{(x_p, y_p)} f\big(g(x_p), g(y_p)\big) +
          1,\\ 
    & = &  f\big(g(x)-1, g(y)\big) +  f\big(g(x), g(y)-1\big) + f\big(g(x)-1, g(y)-1\big) + 1,
  \end{eqnarray*}
  where existence of the suprema is required by the construction of
  the surreals. 
\end{proof}

Functions defined by the recurrence relationship of
\autoref{thm:birthday_mult} have been studied by
Fredman~\cite{fredman82}, and are summarised in
\cite{oeis_a047662}. \autoref{tab:gen_mult} shows values that have
been derived empirically via multiplication of surreal forms, and (in
grey) values that have been derived from the
recurrence. Unfortunately, as the depth of recursion is given by the
birthday, and multiplication is built from recursive application of
multiple recursive operations, combined across the two surreal forms,
we have not been able to pursue complicate surreal multiplications
with resulting generations beyond around 100, though it is noteworthy
that many of the empirical values in the table were first extrapolated
using the recursion, before being verified computationally.

\autoref{thm:birthday_mult} leads to a number of immediate corollaries
regarding asymptotic growth of surreals birthdays in particular
cases. 

\begin{corollary}
  The birthday/generation of $\overline{n}^2$ takes values
  $0,1,6,31,160,841,4494,\ldots $, which grow as  
  \vspace{-2mm}
  \[ g\left(\overline{n}^2\right) \sim a \lambda^n / \sqrt{n}, \]
  where $\lambda = 3 + 2 \sqrt{2} \approx 5.83$ and $a =
   2^{-9/4} \sqrt{\lambda/\pi} \approx 0.29$. 
\end{corollary}

\begin{proof}
  The values $g\left(\overline{n}^2\right)$ are the main diagonal of
  $f(n,m)$, which are given in \cite{oeis_a047665,fredman82}.
\end{proof}

\begin{corollary}
  The birthday/generation of powers of $\overline{2}$ takes values
  $2,6,42, 18006, \ldots$, which grow as
  $ \lfloor c^{(2^n)} \rfloor$ for
  $c = 1.597910218031873178338070118157... $.
\end{corollary}
 
\begin{proof}
  The generation of $\overline{2}^n$ is
  \[
    g\left( \overline{2} \times \overline{2}^{n-1} \right)
    = f\left( 2, g\big( \overline{2}^{n-1} \big) \right) , 
  \]
  where $f(2,n) = n(n+1)$ from \autoref{thm:birthday_mult}, and hence
  \[
    g\big( \overline{2}^{n} \big)
    = g\left( \overline{2}^{n-1} \right) \left( g\big( \overline{2}^{n-1} \big) +1 \right).
  \]
  Now the sequence $g_n = g_{n-1} (g_{n-1} + 1)$ is known~\cite{oeis_a007018},
  and follows  $g_n = \lfloor c^{(2^n)}   \rfloor .$
\end{proof}

The outstanding feature of both of these corollaries is the very high
rate of growth. The size and complexity of calculations involving
these forms grows even faster than the generation of the output due to
the complicated set of recursive operations built on top of each
other, and hence the complexity of larger computations.

\section{Conclusion}
  
This paper derived rules for birthday arithmetic (in particular,
addition, subtraction and multiplication) for surreal number forms.

The notable absentee from this list is division. Naively, division is
the reciprocal of multiplication, and so should be no harder, \ie
$x/y = x \times (1/y)$.  But division is in fact quite different. Only
dyadic surreals have finite representations. Thus we can represent
1/2, 15/16 and so on exactly with a finite form. However, non-dyadic
real and even simple rationals are not finite. Thus numbers even as
simple as 1/3 do not have a finite representation. So to apply the
multiplication theorem to division, we need a formula to calculate the
birthday of a reciprocal. This remains to be found.

\section*{Acknowledgements}

I would like to thank David Roberts for helpful comments on this
manuscript.

    
{
\bibliographystyle{IEEEtran}
\bibliography{paper} 
}\par\leavevmode 

\end{document}